\documentclass[oneside,leqno]{amsart}
 \usepackage{geometry}
 \usepackage{amsmath}
 \usepackage{amssymb} 
 \usepackage{amsthm}
 \usepackage{mathrsfs}
 \usepackage{amscd}
 \usepackage{graphicx}
 \usepackage{enumerate}
 \usepackage{float}
\usepackage{hyperref}
\usepackage{tikz-cd}
\usepackage{mathtools}
\usepackage{color}
\usetikzlibrary{shapes}
\usepackage{comment}
\usepackage{amsmath}
\usepackage{cleveref}

\title{Invertible spectra of finite type}

\author{Guchuan Li}
\address{Department of Mathematics, University of Michigan}
\email{guchuan@umich.edu}

 \newtheorem{thm}{Theorem}[section]
  \newtheorem{prop}[thm]{Proposition}
 
 \newtheorem{lem}[thm]{Lemma}
 \newtheorem{cor}[thm]{Corollary}
 \numberwithin{equation}{section}
  \newtheorem{que}[thm]{Question} 
  \newtheorem*{thm*}{Theorem}

\theoremstyle{definition}
 \newtheorem{defn}[thm]{Definition} 
  
\theoremstyle{remark}
\newtheorem{rem}[thm]{Remark}
 \newtheorem{ex}[thm]{Example}
\newtheoremstyle{case}{}{}{}{}{}{:}{ }{}
\theoremstyle{case}
\newtheorem{case}{Case}

\def\Z{\mathbb{Z}}
\def\G{\mathbb{G}_2}
\def\R{{\mathfrak{Z}}}
\DeclareMathOperator{\holim}{holim}
\DeclareMathOperator{\limone}{lim^1}
\DeclareMathOperator{\Gal}{Gal}
\newcommand{\id}{\mathrm{id}}

\newcommand{\Zp}{\mathbb{Z}_p}

\begin{document} 
\renewcommand{\arraystretch}{2}
\maketitle

\begin{abstract}
We describe the necessary and sufficient numerical condition when an element $X$ in the Picard group of $K(2)$-local category at prime $p \geqslant 5$ is of finite type, i.e., $\pi_kX$ is finitely generated as a $\Z_p$-module for all $k \in \Z$. 
\end{abstract}

%\tableofcontents

%%%%%%
\section{Introduction}
A long-standing open question in chromatic homotopy theory is whether the homotopy groups of $K(n)$-local sphere are finitely generated over $\Zp$ in each degree (\cite{MR2458150}, Conjecture 6.5 in \cite{handbook} Chapter 5). This is the first problem in the \href{https://www-users.cse.umn.edu/~tlawson/hovey/morava.html}{Morava $K$- and $E$-theory} section of Mark Hovey's algebraic topology problem list\footnote{https://www-users.cse.umn.edu/~tlawson/hovey/morava.html}.  A positive answer would follow from the chromatic splitting conjecture \cite{BBP} \cite{handbook}.  One could ask the same question for any invertible spectrum $X$ in $K(n)$-local category.  At height 2, prime $p \geqslant 5$, Hovey and Strickland show that this is not true for most elements in the Picard groups \cite[Proposition~15.7]{MR1601906}.  %\footnote{Strictly speaking, an element in the Picard group is an isomorphism class.  For the purpose of studying homotopy groups, we do not distinguish isomorphic objects.}
From now on, $p \geqslant 5$ is a fixed prime and all spectra are $K(2)$-local.  Let $Pic_{K(2)}$ be the Picard group of $K(2)$-local category.  
We say a $p$-local spectrum $X$ is of finite type if $\pi_k X$ is finitely generated as a $\Zp$-module in all degrees $k$ in $\Z$.  A natural question is:
\begin{que} \label{question}
Which elements of $Pic_{K(2)}$ are of finite type?
\end{que}
	%%Let $Pic_{K(2)}$ be the Picard group of $K(2)$-local category and $X$ be an element in $Pic_{K(2)}$. Is there a criteria when $\pi_* X$ is finitely generated as a $\Zp$-module for all degree $* \in \Z$? %%In another word, if we consider the Picard group grading homotopy groups of $K(2)$-local spheres, for which $\alpha \in Pic_{K(2)}$, $\pi_\alpha S^0$ is finitely generated (for all $\alpha+\Z$)?

Hovey and Strickland's argument is not constructive and does not give a criterion for answering this question. We answer Question \ref{question} by the following necessary and sufficient condition.  It is a condition on coefficients in a ``$p$-adic''-like expansion of a number associated to $X$.

Let $Pic_{K(2)}^0$ be the index $2$ subgroup of invertible elements whose $(E_2)_*$ homology is zero in odd degrees.  Since $X$ is of finite type if and only if $\Sigma X$ is. It is sufficient to analyze which elements of $Pic_{K(2)}^0$ are of finite type. In Corollary \ref{cor:e}, we give a concrete description of the map \cite[Page 11]{Goerss}
\[
e \colon Pic_{K(2)}^0 \rightarrow H^1(\mathbb{G}_2,(E_2/p)^\times_0) \cong \lim_k \Z/p^k(p^2-1).
\]
We denote the $\Zp$-module $\displaystyle\lim_k \Z/p^k(p^2-1) \cong \Zp \oplus \Z/(p^2-1)$ by $\R$. In Section \ref{subsec:alpha}, we will see that $\R$ is a direct summand in $Pic^0_{K(2)}.$
%Then $H^1(\mathbb{G}_2,(E_2/p)^\times_0)\cong \lim_k \Z/p^k(p^2-1)$ is the index $2$ subgroup of $\R$ and the map $e^0$ extends to a group morphism
	\begin{defn}{\label{padicexpansion}}
	For an element $\alpha \in \R$, we define the $\R$-expansion of $\alpha$ to be the unique expansion of the following form 
	\[
	\alpha = a_0 + (p^2-1) \sum^{\infty}_{i=1} a_ip^{i-1} \,\,\, \mbox{for $a_i \in \Z$, 
	$0 \leqslant a_0 < p^2-1$, $0 \leqslant a_i < p$}.
	\]
	We call $a_i$ the $i$\textsuperscript{th} $\R$-expansion coefficients.
	\end{defn}

\begin{thm*}[\Cref{result}]{\label{thm:mainresult}} Let $X \in Pic_{K(2)}^0$.  Then $X$ is of finite type if and only if the $\R$-expansion coefficients $\{a_i\}$ of $e(X)$ contain only either finitely many zero entries or finitely many nonzero entries.
\end{thm*}

\begin{rem}
Note that $e(X)$ contains only finitely many nonzero entries if and only if $e(X)$ is a nonnegative integer.
\end{rem}
Theorem \ref{thm:mainresult} gives a refinement to the following theorem due to Hovey and Strickland.

\begin{thm}{\cite[Theorem~15.1]{MR1601906}}
Let $\mu$ be the unique translation-invariant measure on $\R$ with $\mu(\R)=1$.  We call $\mu$ the Haar measure on $Pic_{K(2)}^0$.  Then the set 
$$\{X \in \R \subset Pic_{K(2)}^0 \mid \text{$X$ is of finite type}\}$$
has Haar measure $0$ in $Pic_{K(2)}^0$.
\end{thm}

We will say $\alpha \in \R$ has the finiteness property if $\alpha$ satisfies the condition in Theorem \ref{thm:mainresult}; that is, there are either finitely many zero entries or finitely many nonzero entries in $\alpha$'s $\R$-expansion coefficients. 

The $\R$-index $e(X)$ has the following beneficial properties.  Let $X$ be a spectrum.  We denote the homotopy cofiber of $X\xrightarrow{p}X$ by $X/p$.

\begin{thm*}[Theorem \ref{thm:modulep}]
Let $X$, $Y$ be elements in $Pic_{K(2)}^0$. Then $X/p \simeq Y /p$ if and only if $e(X)=e(Y)$.
\end{thm*}

\begin{thm*}[Theorem \ref{thm:ghdual}]\label{thm:dual}
Let $X$ be an element in $Pic_{K(2)}^0$, $I_2X$ be the Gross--Hopkins dual of $X$  (see Definition \ref{defn:GrossHopkins}), and $\lambda=\displaystyle\lim_k p^{2k}(p+1)\in \R$. Then $e(I_2X)=1+\lambda-e(X)$.  In particular, $X$ is of finite type if and only if $I_2X$ is of finite type.
\end{thm*}
\begin{rem}
If $e(X)$ is an integer, then $e(I_2X)$ will not be an integer and its $\R$-expansion coefficients contain only finitely many zero entries.

\end{rem}
%Theorem \ref{thm:dual} implies that $I_2X$ satisfies the finitely generated property if and only if $X$ does.  %%This fact also follows easily from the chromatic splitting conjecture at prime $p \geqslant 5$ and height $2$.\textcolor{red}{chromatic splitting applies for all elements or just the sub thick category}

The paper is organized as follows: in section \ref{picardgroup}, we first review some facts about $Pic_{K(2)}^0$ and then give concrete constructions of some elements in it.  In section \ref{reducesection}, we reduce the problem of finite type to those elements with concrete constructions.  In section \ref{computationsection}, we prove our result by computations based on the constructions and the known computation of $\pi_*L_{K(2)}S^0$.  In section \ref{examplesection}, we apply our result to three interesting examples $S^0$, $I_2$, and $S^{2\gamma}$ (to be defined later). 

Through out this paper, $p \geqslant 5$ is a fixed prime and all spectra are $K(2)$-local and group cohomology of $\mathbb{G}_n$ are continuous cohomology.  We may omit $L_{K(2)}$; that is, when we write $S^0$, it means $L_{K(2)} S^0$.

%%%%%%
\subsection*{Acknowledgments}
I would like to heartily thank Paul Goerss for many helpful conversations and the feedback on early drafts of this paper.  I would like to thank Tobias Barthel for explaining how the chromatic splitting conjecture implies that the $K(n)$-local sphere is of finite type.

%%%%%%
\section{The Picard groups of $K(2)$-local categories at prime $p \geqslant 5$}
\label{picardgroup}
The Picard groups of $K(n)$-local categories is introduced by Hopkins (\cite{MR1232198}; see also \cite{MR1263713}). 
\begin{defn}{\cite[Definition 1.2]{MR1263713}}
A $K(n)$-local spectrum $Z$ is \emph{invertible in the $K(n)$-local category} if and only if there is a spectrum $W$ such that
$$L_{K(n)}(Z \wedge W)= L_{K(n)} S^0.$$
The Picard group of $K(n)$-local category $Pic_{K(n)}$ is the group of isomorphism classes of such spectra, with multiplication given by
$$(X, Y) \rightarrow L_{K(n)}(X \wedge Y).$$
\end{defn}
The Picard group of $K(2)$-local categories at prime $p \geqslant 5$ has been computed by Olivier Lader in his thesis, who attributes the result to Hopkins and Karamanov. 	

\begin{thm}{\cite[Theorem ~ 5.3]{Lader}}
		At height $2$, prime $p \geqslant 5$, the Picard group of $K(2)$-local category $Pic_{K(2)}$ is isomorphic to $\Z_p \oplus \Z_p \oplus \Z/ 2(p^2-1)$.
\end{thm}

We explore the algebraic structure of $Pic^0_{K(2)}$ (see \cite{Goerss} and \cite{MR1601906}).  Most parts work for general heights $n$ but we focus on the height $2$ case here. Recall that $\mathfrak{Z}$ denotes $\displaystyle\lim_k \Z/p^k(p^2-1).$
The group $Pic^0_{K(2)}$ is a continuous $\R$-module.  In particular, given $X \in Pic^0_{K(2)}$ and $\alpha \in \R$, there is an element $X^{\alpha} \in Pic^0_{K(2)}$.

The power series $exp(p-)=\displaystyle\sum^{\infty}_{n=0}p^nx^n/n!$ converges for $p>2$.  We have a short exact sequence
$$0\xrightarrow{\hspace*{1cm}} E_0 \xrightarrow{exp(p-)} E_0^\times \xrightarrow{quotient} (E_0/p)^\times \xrightarrow{\hspace*{1cm}} 1.$$
This gives a long exact sequence
\begin{equation}{\label{equation:les}}
\cdots \rightarrow H^1(\mathbb{G}_2, E_0) \xrightarrow{exp(p-)} H^1(\mathbb{G}_2, E_0^\times) \xrightarrow{e} H^1(\mathbb{G}_2, (E_0/p)^\times) \longrightarrow \cdots.
\end{equation}
At $p \geqslant 5$, height $n=2$, we have
$$Pic^0_{K(2)}=H^1(\mathbb{G}_2, (E_2)_0^\times).$$
The quotient map
$$(E_2)_0^\times \rightarrow (E_2/p)_0^\times$$
induces the map
$$e \colon Pic_{K(2)}^0 \cong H^1(\mathbb{G}_2, (E_2)_0^\times) \rightarrow H^1(\mathbb{G}_2, (E_2/p)_0^\times) \cong \mathfrak{Z}.$$
It turns out that the part in \ref{equation:les} becomes a short exact sequence.  We can rewrite it as
\begin{equation}\label{equation:ses}
0 \rightarrow \Z_p \rightarrow Pic^0_{K(2)} \xrightarrow{e} \mathfrak{Z} \rightarrow 0
\end{equation}
This is an exact sequence of continuous $\R$-modules.  The first term $\Zp$ is generated by an element $\zeta \in H^1(\mathbb{G}_2, E_0)$ defined as follows.
\begin{defn}{\cite[Section 1.3]{MR2183282}}\label{definition:zeta}
The homomorphism $\zeta$ is the composition
$$\mathbb{G}_2 \cong \mathbb{S}_2 \rtimes \Gal(\mathbb{F}_{p^2}/\mathbb{F}_p) \xrightarrow{(det, 0)} \mathbb{Z}_p^\times \cong \mathbb{Z}_p.$$
\end{defn}
\begin{rem}
We actually define an element $\zeta$ in $H^1(\mathbb{G}_2, \mathbb{Z}_p)$.  We will denote its image in $H^1(\mathbb{G}_2, E_0)$ also by $\zeta$.  Note that $g'(0)$ originally lives in $\mathbb{F}_{p}^\times$ and we denote its Teichmuller lift to $\mathbb{Z}_p^\times$ also by $g'(0)$.  Then 
by choosing the isomorphism
\begin{align*}
\mathbb{Z}_p^\times &\cong \mathbb{Z}_p \\
1+px &\rightarrow \frac{1}{p}log(1+py),
\end{align*}
a concrete formula of $\zeta \in H^{1,0}(\mathbb{G}_2,(E_2)_0)$ is
\begin{align}\label{equation:zeta}
\mathbb{G}_2 &\rightarrow \mathbb{Z}_p \subset E_0\\ \notag
g &\rightarrow \frac{1}{p}log(g'(0)^{-(p+1)}det(g)).
\end{align}
\end{rem}
There is a splitting map
\begin{align*}
\R &\hookrightarrow Pic^0_{K(2)} \\
\alpha &\rightarrow S^{2\alpha}.
\end{align*}
We shall see a concrete construction of $S^{2\alpha}$ in the Section \ref{subsec:alpha} due to \cite{MR1263713}.  As a continuous $\R$-module, $Pic^0_{K(2)}$ is generated by two topological generators $S^2$ and $S[det]$ with one relation.  For the definition of the generator $S[det]$, see \cite{BBGS}.  The map $det \colon \mathbb{G}_n \rightarrow \Z_p^\times$ is defined in \cite[Section~~ 1.3]{MR2183282}, see also \cite[Section 3]{BBGS}.  The generator $S^2$ can also be realized as a crossed homomorphism $t_0 \colon \mathbb{G}_2 \rightarrow E_0^\times$.  We denote $(S[det])^{\beta}$ by $S^\beta[det]$ for $\beta \in \R$.  Now each element $X\in Pic^0_{K(2)}$ can be presented as $S^{2\alpha} \wedge S^\beta [det]$ where $\alpha,\,\beta \in \R$.  To describe the relation, we introduce $\gamma = \displaystyle\lim_k p^{2k} \in \R$, a generator of the torsion part $\Z/(p^2-1)$ in $\R$, and $\lambda = \displaystyle\lim_k p^{2k}(p+1) \in \R$.  The relation is
\begin{equation}\label{eq:relation}
S^{2\gamma \lambda} = S^\gamma[det].
\end{equation}
The relation follows from knowing the image of $\zeta$ in the short exact sequence \ref{equation:ses}.  This is explained in Proposition \ref{prop:alg}.
\begin{prop}{\cite[Proposition~3.9]{Goerss}}\label{prop:alg}
The image of $\zeta$ is
$$exp(p\zeta)=t_0^{-\lambda}det.$$
\end{prop}
We copy the proof from \cite{Goerss} for completeness.
\begin{proof}
We exam the diagram
\begin{center}
\begin{tikzcd}
  H^1(\mathbb{G}_2,\mathbb{Z}_p) \arrow[r, "exp(p-)"] \arrow[d]
    & H^1(\mathbb{G}_2,\mathbb{Z}_p^\times) \arrow[d] \\
  H^1(\mathbb{G}_2,E_0) \arrow[r, "exp(p-)"]
&H^1(\mathbb{G}_2,E_0^\times) \end{tikzcd}
\end{center}
Plugging in (\ref{equation:zeta}), we have $exp(p\zeta)$ as a crossed homomorphism
\begin{align*}
\mathbb{G}_2 &\rightarrow \mathbb{Z}_p^\times \subset E_0^\times \\
g &\rightarrow g'(0)^{-(p+1)}det(g).
\end{align*}

Note that $t_0(g)^\lambda=g'(0)$ mod $m$.  Let $\gamma$ be $\displaystyle\lim_k p^{2k}\in \R$.  Then $t_0(g)^{p^{2k}}=g'(0)^{p^{2k}}\text{ mod }m^{p^{2k}}$ and we have $t_0^\gamma=g'(0)^\gamma=g'(0)$.  Also note that $(p+1)\gamma=\lambda$, the image of $\zeta$ is 
$t_0^\lambda det(g)$.
\end{proof}

\begin{cor}\label{cor:e}
Let $X \in Pic_{K(n)}^0$ be $S^{2\alpha} \wedge S^\beta[det]$ where $\alpha, \, \beta \in \R$.  Then $e(X)=\alpha+\lambda \beta$.
\end{cor}
\begin{proof}
By Proposition \ref{prop:alg} and the exactness of \ref{equation:ses}, the kernel of the map $e$ is generated by $S^{-2\lambda}S[det]$.  
Therefore, we have
$$e(S^{2\alpha}\wedge S^\beta[det])=e(S^{2\alpha})+e(S^{2\lambda\beta})=\alpha+\lambda\beta.$$
\end{proof}

%\begin{rem}
%For the meaning of $S^\beta[det]$, here we use the $\R$-module structure on $Pic_{K(2)}$.  Hovey and Strickland also define $P^{(k)}$ for a $K(n)$-local spectrum $P$ with $K(n)_*P =K(n)_*$ and $k \in \Zp$ (up to canonical isomorphism) (see \cite[Corollary~14.9]{MR1601906}).
%\end{rem}

\subsection*{Construction of $S^{2\alpha}$}\label{subsec:alpha}

This appeared in \cite{Goerss} and \cite{MR1263713}.  The construction works for all heights.  Here we focus on the height $2$ case.  There is a construction of $S^{2\alpha} \in Pic_{K(2)}^0$ for a given $\alpha \in \R$.

	We introduce some notation for the $\R$-expansion.
	\begin{defn}{\label{padic}}
	Let $\alpha \in \R$ with the expansion
	\[
	\alpha = a_0 + (p^2-1) \sum^{\infty}_{i=1} a_ip^{i-1} \,\,\, \mbox{for $a_i \in \Z$, \, $0 \leqslant a_0 < p^2-1$, \,$0 \leqslant a_i < p$}.
	\]
	Denote 
	\[
	\sum^{\infty}_{i=1} a_ip^{i-1} \in \mathbb{Z}_p
	\]
	by $\bar{\alpha}$.
	For $k \geqslant 0$, define $\alpha_k \in \Z$ to be
	\[
	\alpha_k = a_0 + (p^2-1)\sum^{k}_{i=1} a_i p^{i-1}
	\]
	and for $k \geqslant 1$, define $\bar{\alpha}_k \in \Z$ to be
	\[
	\bar{\alpha}_k = \sum^{k}_{i=1} a_i p^{i-1}.
	\]
	\end{defn}	
		
	We also need $v_1$-self maps in the construction.  A $v_1^{\ell}$-self map of $X$ is a map 
	$$f \colon\Sigma^{\ell|v_1|}X \rightarrow X$$
such that $K(1)_*(v_1^{\ell})$ is given by multiplying $v_1^{\ell}$ (see \cite{MR1192553}).  We work in large primes so $S^0/p$ has a unique $v_1$-self map. We will abuse notation and write $v_1^\ell$ for powers of this unique map.
	
	We will construct $S^{2\alpha}$ as a homotopy limit of generalized Moore spectra. The generalized Moore spectra are constructed by Hopkins and Smith (\cite{MR1652975}), explained and discussed in \cite[Chapter~6]{MR1192553} and \cite[Section~4]{MR1601906}.  Here we will follow the notation in \cite{MR1601906} and make specific choices for our purposes.  Let $S^0/p$ denote the cofiber of  $p \colon S^0 \rightarrow S^0$.  In general, $S^0/p^k$ denotes the cofiber of $p^k \colon S^0 \rightarrow S^0$. If $S^0/p^k$ admits a $v_1^{\ell}$-self map, let $S^0/(p^k,v_1^{\ell})$ denote the cofiber of $v_1^{\ell} \colon \Sigma^{\ell|v_1|} S^0/p^k \rightarrow S^0/p^k$.
	
	Recall that from the computation of $K(2)$ local spheres at $p \geqslant 5$, $S^0/p^{k+1}$ admits a $v_1^{p^k}$-self map and hence $S^0/(p^{k+1},v_1^{p^k})$ exists. Also $S^0/(p^{k+1},v_1^{p^k})$ admits a $v_2^{p^k}$-self map.  These $v_2^{p^k}$-self maps are weak equivalence (in $K(2)$-local category) so we can make the following definition.
		
	\begin{defn}\label{exp}
		Given $\alpha \in \R$, $S^{2\alpha}$ is defined as $\holim \Sigma^{2\alpha_k} S^0/(p^{k+1},v_1^{p^k})$. The maps in the inverse system are
		$$\Sigma^{2\alpha_{k+1}} S^0/(p^{k+2},v_1^{p^{k+1}}) \overset{q}{\rightarrow} \Sigma^{2\alpha_{k+1}} S^0/(p^{k+1},v_1^{p^k}) \overset{v_2^{a_{k+1}p^{k}}}{\rightarrow} \Sigma^{2\alpha_k} S^0/(p^{k+1},v_1^{p^k}),$$
		where the first map $q$ is the quotient and the second map is the $v_2^{a_{k+1}p^{k}}$-self maps.
	\end{defn}
The construction of $S^{2\alpha}$ gives the splitting map
\begin{align*}
	g \colon R &\rightarrow Pic_{K(2)}^0 \\
	\alpha &\rightarrow S^{2\alpha}.
\end{align*}
such that $e \circ g = \id_{Pic_{K(2)}^0}$. In particular, the group morphism $g$ is injective.
%\begin{prop}
%	The group morphism
%	\begin{align*}
%	g \colon \lim_k \Z/2p^k(p^2-1) &\rightarrow Pic_{K(2)} \\
%	\alpha &\rightarrow S^\alpha
%	\end{align*}
%	is injective.
%\end{prop}

%%%

%%%%%%
\section{Reduction modulo $p$}\label{reducesection}	
%%	We do this computation as a warm up, it is similar to height 1 case, \textcolor{blue}{should do K(1) case as a warm up too.}
%%	\begin{thm}
%%		Let $X \in Pic_{K(2)}$, then $\pi_* X \otimes \mathbb{Q}=...$. \textcolor{blue}{$\otimes \Z_p?$}
%%	\end{thm}
	
%%	\textcolor{blue}{projective module methods, which works for $S^n$, for general $X \in Pic_{K(2)}$, we need a little computation of the central part, should be a height 1 warm up computaiton}

%%	\begin{cor}
%%		\textcolor{blue}{Let $X \in Pic_{K(2)}$, then $\pi_* X$ is finite except for finite stems if and only if $\pi_* X/p$ is finite.}
%%	\end{cor}

%%%
	In this section, we show that $X \in Pic^0_{K(2)}$ has finite type if and only if $\pi_k S^{2e (X)}/p$ is a finite dimensional $\mathbb{F}_p$ vector space for all $k \in \Z$.  The case $X=L_{K(2)}S^0$ is explained in \cite{MR2458150}.
	
	\begin{prop}{\label{fg}}
		Let $X \in Pic_{K(2)}$ and define $X/p$ as $X \wedge S^0/p$. Then $X$ is of finite type if and only if $\pi_k X/p$ is finite for all $k \in \Z$. 		
	\end{prop}

	\begin{proof}	
	The only if part follows from the long exact sequence
	$$\cdots \rightarrow \pi_k X \xrightarrow{p} \pi_k X \rightarrow \pi_k X/p \rightarrow \pi_{k+1}X \rightarrow \cdots.$$
	The if part follows from the fact that $\pi_*X$ is $p$-complete when $\pi_k X/p$ is finite for all $k \in \Z$.  In this case, we can show $\pi_k X/p^i$ is also finite for all $k \in \Z$ inductively by the cofiber sequences
	$$X/p \rightarrow X/p^{i+1} \rightarrow X/p^i.$$
Then the $\lim^1$ term vanishes and we have
	\begin{equation*}
	\pi_*X \cong \lim \pi_*X/p^i.
	\end{equation*}
%This isomorphism factors through $\lim \frac{\pi_*X}{p^i \pi_*X}$ as
%	$$\pi_*X \rightarrow \lim \frac{\pi_*X}{p^i \pi_*X} \hookrightarrow \lim \pi_*X/p^i.$$
%The second map is an injection from the long exact sequence.  This implies that
%$$\pi_*X \cong \lim \frac{\pi_*X}{p^i \pi_*X},$$
%i.e., $\pi_*X$ is $p$-complete. 		
	\end{proof}

Now we focus on $X/p$ for $X \in Pic_{K(2)}$.  The following theorem tells that after reduction modulo $p$, the two topological generators $S^2$ and  $S[det]$ of $Pic_{K(2)}$ become the same up to a $\R$-suspension.
	
	\begin{thm}{\cite{MR1217353}; see also \cite[Theorem~3.11]{Goerss}}\label{thm:relation}
		Let $\lambda = \lim p^{2k}(p+1) \in \R$. Then there is an equivalence
		$$S^{2\lambda}/p = \Sigma^{2\lambda} S^0/p \simeq S[det]/p.$$
	\end{thm}
	
In particular, this implies that $X/p$ is determined by $e(X) \in \R$.
\begin{thm}\label{thm:modulep}
Let $X$, $Y$ be elements in $Pic_{K(2)}^0$. Then $X/p \simeq Y/p$ if and only if $e(X)=e(Y)$.
\end{thm}
\begin{proof}
In Section \ref{picardgroup}, we see that an element $X \in Pic_{K(2)}^0$ can be presented as $S^{2\alpha} \wedge S^\beta[det]$ where $\alpha \, \beta \in \R$.  By Theorem \ref{thm:relation} and Corollary \ref{cor:e}, we have
$$X/p \simeq S^{2\alpha+2\beta\lambda}/p= S^{2e(X)}/p.$$
Therefore, we have $X/p \simeq Y/p$ if and only if $e(X)=e(Y)$.
\end{proof}

	Now for a given element $X \in Pic_{K(2)}^0$, Property \ref{fg} tells us that $X$ is of finite type if and only if $X/p$ is, and Theorem \ref{thm:modulep} implies that $X/p \simeq S^{2e (X)}/p$.  Hence, the question if $X$ is of finite type rededuces to the question if $S^{2e(X)} /p $ is of finite type. We have the following corollary.
	
	\begin{cor}{\label{reduce}}
		Given $X \in Pic_{K(2)}^0$, $X$ is of finite type if and only if $S^{2e (X)}/p$ is of finite type.
	\end{cor}
%%%%%%%
%Section4
%%%%%%%
\section{computation of $\pi_* S^{2\alpha} / p$}\label{computationsection}	
	In this section, we give a computation of $\pi_* S^{2\alpha} / p$ for any $\alpha \in \R$.  Our computation is based on the homotopy limit construction of $\alpha$ spheres and the known computation of $\pi_*L_{K(2)}S^0/p$. The latter is computed by Shimomura and Yabe \cite{MR1318877} and explained by Behrens \cite{MR2914955}.  See also Lader's thesis \cite{Lader} for a more group theoretical computation of $\pi_*L_{K(2)}S^0/p$. (We will often omit the $L_{K(2)}$ in the following discussion.)
	
	We compute $\pi_n S^{2\alpha} / p=\pi_n\holim \Sigma^{2\alpha_k}S^0/(p,v_1^{p^k})$ via the short exact sequence
	$$0\rightarrow \limone \pi_{n+1}\Sigma^{2\alpha_k}S^0/(p,v_1^{p^k}) \rightarrow \pi_n\holim \Sigma^{2\alpha_k}S^0/(p,v_1^{p^k}) \rightarrow \lim \pi_n\Sigma^{2\alpha_k}S^0/(p,v_1^{p^k})\rightarrow 0.$$
	Because $\pi_n \Sigma^{2\alpha_k}S^0/(p,v_1^{p^k})$ is finite for all $n \in \mathbb{Z}$, the $\limone$ term vanishes and
	$$\pi_n S^{2\alpha} / p \cong \lim \pi_n \Sigma^{2\alpha_k}S^0/(p,v_1^{p^k}).$$
	
	We compute $\pi_* S^0/(p,v_1^{p^k})$ via the homotopy fixed point spectral sequence
	$$E_2^{s,t}=H^s (\G, E_t (S^0/(p,v_1^{p^k}))) \Rightarrow \pi_{t-s}(S^0/(p,v_1^{p^k})).$$
	There is no room for differentials and nontrivial extensions for degree reasons because 
	$$E_2^{s,t}=0 ~ when ~ 2(p-1) \nmid t ~ or ~s>4,$$
	$$d_r \colon E_r^{s,t} \rightarrow E_r^{s+r,t+r-1}.$$
	For all nontrivial group $\pi_m L_{K(2)}S^0/(p,v_1^k)$, there is a unique pair $(s,t)$ such that $m=t-s$ and 
$$\pi_m L_{K(2)}S^0/(p,v_1^k) \cong H^s(\mathbb{G}_2, E_t/(p,v_1^k)).$$
Therefore, we will not distinguish elements in the $E_2$ page and elements in the homotopy groups.

%%%	
\begin{comment}
	We compute $\pi_* S^{2\alpha} / p$ via the homotopy fixed point spectral sequence
	$$E_2^{s,t}=H^s (\G, E_t (S^\alpha / p)) \Rightarrow \pi_{t-s}(S^\alpha / p).$$
	It collapses for degree reasons because 
	$$E_2^{s,t}=0 ~ when ~ 2(p-1) \nmid t ~ or ~s>4,$$
	$$d_r \colon E_r^{s,t} \rightarrow E_r^{s+r,t+r-1}.$$
	By the definition of $S^\alpha$, we have 
	$$H^s (\G, E_t (S^\alpha / p)) = H^s (\G, E_t \holim (S^{\alpha_k} / (p, v_1^{p^k}))).$$
	Because $E_t(S^{\alpha_k}/(p, v_1^{p^k}))$ is finite for all $k$, the $lim^1$ term vanishes, and we have \textcolor{red}{check this, something in Devinatz-Hopkins, also it is formal that continuous cohomology commutes with limit; It might be easier to prove that $\pi_*S^{2\alpha}/p \cong \lim \pi_*\Sigma^{2\alpha_k} S^0/(p,v_1^{p^k})$}
	$$H^s (\G, E_t (S^\alpha / p)) \cong H^s (\G, \lim E_t (S^{\alpha_k} / (p, v_1^{p^k}))) \cong \lim H^s (\G, E_{t-\alpha_k}/ (p, v_1^{p^k})).$$
\end{comment}
%%%

%%%%%%

We list the result of $\pi_*S^0/p$ below as Theorem \ref{computation}.  We need to compute the map in the inverse limit
	$$f_{k+1} \colon \Sigma^{2 \alpha_{k+1}} S^0/(1,p^{k+1}) \overset{q}{\rightarrow} \Sigma^{2 \alpha_{k+1}} S^0/(1,p^k) \overset{v_2^{a_{k+1}p^{k}}}{\rightarrow} \Sigma^{2 \alpha_k} S^0/(1,p^k).$$
	where $\alpha=a_0+a_12(p^2-1)+a_2 2(p^2-1)p+\cdots$. Algebraically this is the quotient map composed with the multiplication of $v_2^{a_{k+1}p^{k}}$
	\begin{align*}
 	\pi_*f_{k+1} \colon H^s(\mathbb{G}_2,E_*\Sigma^{2 \alpha_{k+1}} S^0/(1,p^{k+1}) & \rightarrow H^s(\mathbb{G}_2,E_* \Sigma^{2 \alpha_k} S^0/(1,p^k))\\
 	x &\rightarrow v_2^{a_{k+1}p^{k}}x.
 	\end{align*}

	Therefore, the computation of $\pi_* S^\alpha / p$ reduces to the computation of the limit, which we will explain in this section.  We will need some computational results about $\pi_* S^0/(p, v_1^{p^k})$.  We follow the notation in \cite{MR2914955}, for readers who are familiar with Shimomura's notations, there is a dictionary in \cite{MR2914955} between the names.  The algebraic descriptions are convenient for limit computations, but pictures of the result are much easier to digest, the author encourages readers to see figures in \cite{MR2914955}.  In this section, all the differentials are really $v_1$-Bockstein spectral sequence differentials.  In particular, we only list the leading term with respect to $v_1$-power in the formula.  For example, the differential
	$$d v_2^{sp^n} = v_1^{b_n}v_2^{sp^n-p^{n-1}} h_0$$
really means
	$$d v_2^{sp^n} = a v_1^{b_n}v_2^{sp^n-p^{n-1}} h_0+\text{higher $v_1$ terms}$$
	for some $a \in \mathbb{F}_p^\times$ in the $v_1$-Bockstein spectral sequence.

\begin{defn}
Let $A$ be the $\mathbb{F}_p-$algebra generated by $6$ elements $\{1, h_0,h_1,g_0,g_1, h_0g_1=h_1g_0\}$.  We denote this basis as $B$.  We assign bidegrees (homological degree, internal degree) to elements in $A$ as follows
\begin{align*}
|1|&=(0,0),           &  |h_0|&=(1, 2(p-1)),              &  |h_1|&=(1, -2(p-1)), \\
|g_0|&=(2, 2(p-1)),         &  g_1&=|2, -2(p-1)|,   &  |h_0g_1|&=(3,0).
\end{align*}
\end{defn}

\begin{thm}{\cite[Theorem 3.2]{MR0431168}}
We have
$$H^s(\mathbb{G}_2, (E_2/(p, v_1))_t) = \mathbb{F}_p[v_2^{\pm 1}]\otimes A \otimes \Lambda[\zeta].$$
\end{thm}
Let $\mathbb{G}_2^1$ denote the kernel of the homomorphism in (\ref{definition:zeta}) $\zeta \colon \mathbb{G}_2 \rightarrow \mathbb{Z}_p$.
Then $\mathbb{G}_2=\mathbb{G}_2^1 \rtimes \mathbb{Z}_p$
and 
$$H^s(\mathbb{G}_2^1, (E_2/p)_t) = H^s(\mathbb{G}_2^1, (E_2/p)_t) \otimes \Lambda[\zeta].$$
\begin{thm}[\cite{MR1318877},\cite{MR2914955}]{\label{computation}}
		There exists a complex $(C_0,d)$ such that $H^*(C_0,d)=H^*(\mathbb{G}_2^1, (E_2/p)_*)$ where $C_0 := A \otimes \mathbb{F}_p[v_2^{\pm 1}] \otimes \mathbb{F}_p[v_1] $ and differentials  given in table \ref{diff0} are $v_1$ linear. The $\mathbb{F}_p-$generators are listed in table \ref{element0}.
		
		Let $C$ be $C_0 \otimes \Lambda(\zeta)$. Then we have
		$$H^*(C,d)=H^*(\mathbb{G}_2, (E_2)_*/p).$$
	\end{thm}

			\begin{table}[H] 
			\centering
			\caption{differentials in $C_0$ \cite[Section 4]{MR1318877}, \cite[Theorem 3.2]{MR2914955}}
			\label{diff0}
			\begin{tabular}{| l | l |}
			\hline
 				$d v_2^{sp^n} = v_1^{b_n}v_2^{sp^n-p^{n-1}} h_0$	&  	$p \nmid s, n \geqslant 1$	\\
			\hline
				$d v_2^s = v_1 v_2^s h_1$		&  	$p \nmid s$	\\
			\hline
				$d v_2^{sp^n} h_0 =  v_1^{A_n + 2} v_2^{sp^n - \frac{p^n-1}{p-1}} g_1$	&	$ s \not\equiv 0, -1 ~ mod ~ p, ~ n \geqslant 0$	\\
			\hline
				$d v_2^{sp^n - p^{n-2}} h_0 = v_1^{p^n - p^{n-2} + A_{n-2} + 2} v_2^{sp^n - p^{n-1} - \frac{p^{n-2} - 1}{p-1}} g_1 $	 &	$ \forall s, ~ n \geqslant 2$ \\
			\hline
 				$d v_2^{sp} h_1 = v_1^{p-1} v_2^{sp - 1} g_0$	&	$ \forall s$  \\
			\hline
 				$d v_2^{sp^n - \frac{p^{n-1}-1}{p-1} } g_1 = v_1^{b_n} v_2^{sp^n - \frac{p^n - 1}{p-1}} h_0 g_1$	&	$ s \not\equiv -1 ~ mod ~ p, ~ n \geqslant 1$ 	\\
			\hline
 				$d v_2^s g_0 = v_1 v_2^s h_0 g_1$	&	$ s \not\equiv -1 ~ mod ~ p$\\
			\hline				
			\end{tabular}
			\end{table}

We follow the notation in \cite{MR2914955} to define $b_n$ and $A_n$ as follows
$$b_n=
\begin{cases}
p^{n-1}(p+1)-1, \,\,\, n\geqslant 1,\\
1, \,\,\, n=0;
\end{cases}
$$
$$
A_n=
\begin{cases}
(p^n-1)(p+1)/(p-1), \,\,\, n\geqslant 1,\\
0, \,\,\, n=0.
\end{cases}
$$

			\begin{table}[H] 
			\centering
			\caption{Generators for $H^*(C_0, d)$}
			\label{element0}
			\begin{tabular}{| l |p{5cm} | p{5cm}|}
			\hline
				$v_1^j$		&  	&	$j \geqslant 0 $ \\
				\hline
				$v_1^j h_0$	&	&  	$j \geqslant 0 $	\\
				\hline
				$v_1^j v_2^{sp^n-p^{n-1}} h_0$	&	$ p \nmid s, ~ n \geqslant 1$	&	$ 0 \leqslant j \leqslant b_n - 1 $ \\
				\hline
				$v_2^s h_1$	&	$p \nmid s$	&	\\
				\hline
				$v_1^j v_2^{sp-1} g_0$	&	$\forall s$	&	$ 0 \leqslant j \leqslant  p-2, p^k-1 $ 	\\
				\hline
				$v_1^j v_2^{sp^n - \frac{p^n - 1}{p - 1}} g_1 $	 &	$ s \not\equiv 0, -1 ~ mod ~ p, ~ n \geqslant 0$	&	$ 0 \leqslant j \leqslant A_n + 1 $ \\
				\hline
 				$v_1^j v_2^{sp^n - p^{n-1} - \frac{ p^{n-2}-1}{p-1}} g_1 $	&	$\forall s, ~ n \geqslant 2 $	&	$ 0 \leqslant j \leqslant p^n - p^{n-2} +A_{n-2} +1 $	\\
				\hline
 				$v_1^j v_2^{sp^n - \frac{p^n - 1}{p-1}} h_0 g_1$	&	$ s \not\equiv -1 ~ mod ~ p, ~ n \geqslant 0$ 	&	$ 0 \leqslant j \leqslant  b_n - 1 $	\\
				\hline
			\end{tabular}
			\end{table}
			
			For $S^0/(p, v_1^{p^k})$, we have $H^*(\mathbb{G}_2^1, E_* S^0/(p, v_1^{p^k}) )= H^*(C/v_1^{p^k}, d),$ which can be computed from $H^*(C,d)$ via the long exact sequence
		$$\cdots \rightarrow H^s(C,d) \xrightarrow{\times v_1^{p^k}} H^s(C,d) \xrightarrow{q} H^s(C/v_1^{p^k},d) \rightarrow \cdots.$$
		
The $\mathbb{F}_p$-generators of $H^*(C_0/v_1^{p^k},d)$ are listed in Table~\ref{element1}, where the elements are divided into two subsets: the upper half $Coker$ part (the cokerel of $v_1^{p^k}$), and the lower half $Ker$ part (the kernel of $v_1^{p^k}$).  While the elements in the $Coker$ part have straightforward names, we need to keep track of the boundary connecting morphism to name the elements in the $Ker$ part.  For example, $v_1^jv_2^{sp^n-p^{n-1}}h_0$ with $p \nmid s$, $n \geqslant 1$, $\max\{0,\,b_n-p^k\} \leqslant j \leqslant p^k-1$ is in the kernel of
$$\times v_1^{p^k} \colon H^1(C_0,d) \rightarrow H^1(C_0,d)$$ when $k \geqslant 1$.  Let $\partial$ be the boundary connecting morphism
$$\partial_0 \colon H^0(C_0/v_1^{p^k},d) \rightarrow H^1(C_0,d).$$
in the long exact sequence
$$H^0(C_0,d) \xrightarrow{v_1^{p^k}} H^0(C_0,d) \rightarrow H^0(C_0/v_1^{p^k},d) \xrightarrow{\partial_0} H^1(C_0,d) \xrightarrow{v_1^{p^k}} H^1(C_0,d) \rightarrow \cdots$$

By the snake lemma and the following differential in $C_0$, for $p \nmid s$, $n \geqslant 1$
$$d_0(v_2^{sp^n})=v_1^{b_n}v_2^{sp^n-p^{n-1}}h_0,$$
we have
$$\partial (v_1^{j-b_n+p^k}v_2^{sp^n})=v_1^jv_2^{sp^n-p^{n-1}}h_0.$$
After reindexing the power of $v_1$, we name the corresponding $\mathbb{F}_p-$generators in the $Ker$ part of $H^0(C_0/v_1^{p^k})$ as $v_1^jv_2^{sp^n}$ with $p \nmid s$, $n \geqslant 1$ and $\max\{0,\,p^k-b_n\} \leqslant j \leqslant p^k-1.$  Similarly, the first row in the $Ker$ part comes from the differential 
$$dv_2^s=v_1v_2^sh_1.$$

			Denote the result of $H^*(C_0/v_1^{p^k},d)$ as $X_1$, then $X_1 \otimes \Lambda (\zeta)$ gives the $E_2$ page and in this case also the $E_{\infty}$ page of the ANSS that converges to $\pi_* S^0/(p, v_1^{p^k})$.
			\begin{table}[H] 
			\centering
			\caption{Generators for $H^*(C_0/(v_1^{p^k}), d)$}
			\label{element1}
			\begin{tabular}{| p{3.5cm} | l |p{3.5cm} | p{3.2cm}| p{5.2cm}}
			\hline
 				Names & $Coker$ part		&	&	\\
			\hline
			\hline
				
				$(1, k)$ & $v_1^j$		&  	&	$0 \leqslant j \leqslant p^k-1 $	\\
				\hline
				$(h_0, k)$ & $v_1^j h_0$	&	&  	$0 \leqslant j \leqslant p^k-1 $	\\
				\hline
				$(v_2^{sp^n-p^{n-1}}h_0,k)$ & $v_1^j v_2^{sp^n-p^{n-1}} h_0$	&	$ p \nmid s, ~ n \geqslant 1$	&	$ 0 \leqslant j \leqslant min \{ b_n - 1, p^k - 1 \}$ \\
				\hline
				$(v_2^sh_1,k)$ & $v_2^s h_1$	&	$p \nmid s$	&	\\
				\hline
				$(v_2^{sp-1} g_0, k)$ & $v_1^j v_2^{sp-1} g_0$	&	$\forall s$	&	$ 0 \leqslant j \leqslant min\{ p-2, p^k-1 \}$ 	\\
				\hline
				$(v_2^{sp^n - \frac{p^n - 1}{p - 1}} g_1,k)$ & $v_1^j v_2^{sp^n - \frac{p^n - 1}{p - 1}} g_1 $	 &	$ s \not\equiv 0, -1 ~ mod ~ p, ~ n \geqslant 0$	&	$ 0 \leqslant j \leqslant min\{A_n + 1, p^k - 1 \} $ \\
				\hline
 				 $(v_2^{sp^n - p^{n-1} - \frac{ p^{n-2}-1}{p-1}} g_1,
				 k)$ & $v_1^j v_2^{sp^n - p^{n-1} - \frac{ p^{n-2}-1}{p-1}} g_1$	&	$\forall s, ~ n \geqslant 2 $	&	$ 0 \leqslant j \leqslant min \{ p^n - p^{n-2} +A_{n-2} +1, p^k - 1 \}$	\\
				\hline
 				$(v_2^{sp^n - \frac{p^n - 1}{p-1}} h_0 g_1, k)$ & $v_1^j v_2^{sp^n - \frac{p^n - 1}{p-1}} h_0 g_1$	&	$ s \not\equiv -1 ~ mod ~ p, ~ n \geqslant 0$ 	&	$ 0 \leqslant j \leqslant min \{ b_n - 1, p^k - 1 \}$	\\
 				
				\hline
				\hline

				&$Ker$ part	&	&	\\
				\hline
				\hline

				&$v_1^{p^k - 1} v_2^s$		&	$ p \nmid s$	&	\\
				\hline
				&$v_1^j v_2^{sp^n}$		&	$ p \nmid s, ~ n \geqslant 1$	&	$max \{ 0, p^k - b_n \} \leqslant j \leqslant p^k - 1$ \\
				\hline
				&$v_1^j v_2^{sp^n} h_0 $	&	$ s \not\equiv 0, -1 ~mod ~ p, ~ n \geqslant 0$	&	$ max \{0, p^k - A_n -2 \} \leqslant j \leqslant p^k -1$ \\
				\hline
				&$v_1^j v_2^{sp^n - p^{n-2}} h_0$	&	$\forall s, ~ n \geqslant 2$	&	$max \{0, p^k - (p^n - p^{n-2} + A_{n-2} +2) \} \leqslant j \leqslant p^k - 1$	\\
				\hline
				&$v_1^j v_2^{sp} h_1$		& $\forall s$	&	$max \{0, p^k - p +1 \} \leqslant j \leqslant p^k - 1$\\
				\hline
				&$v_1^{p^k - 1} v_2^s g_0$	&	$ s \not\equiv -1 ~mod ~ p$	&	\\
				\hline
				&$v_1^j v_2^{sp^n - \frac{p^{n-1}-1}{p-1}} g_1$ 	&	$ s \not\equiv -1~mod ~ p,~ n \geqslant 1$	&	$max \{ 0, p^k - b_n \} \leqslant j \leqslant p^k - 1$	\\
				\hline		
			\end{tabular}
		\end{table}
	
	The essential computation is done by Shimomura and Yabe \cite{MR1318877}.  The author learned it from Behrens's paper \cite{MR2914955}.  The idea of organizing the result as a chain complex goes back to Henn, Karamanov and Mahowald.
	
	%We will group elements in $H^*(C_0/v_1^{p^k},d)$ by $v_1$ towers depending on their length and keep track the length of $v_1$ towers in the limit to make a computation of $\pi_*S^\alpha/p$.  
	We introduce some notation before going into the computation of $\displaystyle\lim_k\pi_*\Sigma^{2\alpha_k}S/(p,v_1^{p^k})$.  In Definition \ref{padic}, for an element  $\alpha \in \R$, we have defined a p-adic number $\bar{\alpha}$ and its truncation $\bar{\alpha}_k$.  We use the notation $v_2^{-2\alpha}$ to name elements in the limit $\displaystyle\lim_k\pi_*\Sigma^{2\alpha_k}S/(p,v_1^{p^k})$. 
\begin{defn}\label{rem:name}	
	Let $x \in B$ and
		 \begin{equation}
		 \alpha = a_0 + a_1(p^2-1))+ a_2 p(p^2-1) + a_3 p^2(p^2-1) + \cdots.
		 \end{equation}
		 Then for $m,\ell \in \Z$, define an element 
		 $$v_1^m v_2^{\ell-2\alpha} x \in \displaystyle\lim_k H^{*,*}(S^{2\alpha_k}/ p, v_1^{p^k}) = \lim_k y_k$$
		 by setting
		 \begin{equation}
		 y_k = v_1^m v_2^{\ell-a_1-a_2p-\cdots-a_{k}p^{k-1}} \in H^{*,*-2a_0}(S^0/p, v_1^{p^k})
		 \end{equation}
		 and $y_{k+1} \rightarrow y_k$ is given by multiplying $v_2^{a_{k+1}p^k}$.
		 We will denote $\ell-a_1-a_2p-\cdots-a_{k}p^{k-1}$ by $\ell_k$. 
\end{defn}		
		
		We use the following notation to describe the power of $v_1$.
\begin{defn}
	For $k \in \mathbb{N}$, define the map $J_k$ from $\mathbb{Z} \times R \times B$ to subsets of natural numbers $\mathbb{N}$ by
	$$(l, \alpha, x) \mapsto \{j \in \mathbb{N} \mid v_1^j v_2^{\ell_k} x \neq 0 \in H^*(C_0/(v_1^{p^k}),d)\}.$$
\end{defn}
\begin{thm}
Given $\alpha \in \R$, we have a $\mathbb{F}_p$-basis of $\pi_*S^{2\alpha}/p$ as follows
$$\{\zeta^\epsilon v_1^j v_2^{\ell-2\alpha} x \mid \epsilon=0,1, \ell \in \mathbb{Z}, x\in B, j \in \lim_{k\rightarrow \infty} \inf J_k(\ell, \alpha, x)\}.$$
\end{thm}
\begin{proof}
The above gives a $\mathbb{F}_p$-basis of $\displaystyle\lim_k \pi_*\Sigma^{2\alpha_k} S^0/(p,v_1^{p^k})$ by the definition of $J_k$.  The theorem follows from 
$$\pi_*S^{2\alpha}/p=\lim_k \pi_*\Sigma^{2\alpha_k} S^0/(p,v_1^{p^k})$$
as we discussed in the beginning.
\end{proof}
				
		Note that if $\pi_*S^{2\alpha}/p$ is a finitely generated $\mathbb{Z}_p$-module if and only if $(\pi_*S^{2\alpha}/p)/\zeta$ is a finitely generated $\mathbb{Z}_p$-module.  We will drop $\zeta$ in later analysis. 
				
	\begin{defn}\label{defn:stable1}	
	We say an element $(\ell, \alpha, x) \in \mathbb{Z}\times R \times B$ is \emph{stable} if there exists $K \in \Z^+$ and for all $k>K$, we have $J_k(\ell, \alpha, x)=J_K(\ell, \alpha, x)$.  Otherwise, we say $(\ell, \alpha, x)$ is \emph{unstable}.
	
	Given an $\alpha \in \R$, we say an element $v_1^j v_2^{\ell-2\alpha} x \in \pi_*S^{2\alpha}/p$ is \emph{$\alpha$-stable} if $(\ell,\alpha, x)$ is stable.  Otherwise, we will call the element \emph{$\alpha$-unstable}.
	\end{defn}		
	In particular, if $\displaystyle\lim_{k\rightarrow \infty} \inf J_k(\ell, \alpha, x)=\emptyset$, we say the element $(\ell, \alpha, x)$ is stable to trivial.
	%%%
	\begin{ex}
	When $\ell-\bar{\alpha}=0$, we know $\alpha$ is an integer, then all elements $(\ell, \alpha, x)$ are stable.  When $\ell-\bar{\alpha} = \displaystyle\sum^\infty_{k>0}(p-2)p^k+p-1$, we have interesting $\alpha$-unstable elements.  For example, set $\ell=0$ and $\bar{\alpha}=\displaystyle\sum^\infty_{k\geqslant 0}p^k$.  Then $v_1^jv_2^{-2\alpha}g_1$ and $v_1^jv_2^{-2\alpha}h_0g_1$ are two unstable elements.  Because when $k$ is large, we have $l_k=-\frac{p^k-1}{p-1}$.  Then when $k>0$, we have $J_k(0, \alpha, g_1)=\{\max\{0,p^k-b_k\}=0 \leqslant j \leqslant p^k-1\}$ and $\displaystyle\liminf_{k\rightarrow \infty}J_k(0, \alpha, g_1)=\mathbb{N}$.  This gives an infinite $v_1$-tower on $v_2^{-2\alpha}g_1$. The element $v_1^jv_2^{-2\alpha}h_0g_1$ has similar situations.
	\end{ex}
	%%%
	
	Because $S^{2\alpha}$ is of finite type if and only if $S^{2\alpha-2a_0}$ is of finite type, we can now assume that $a_0=0$, i.e., $\alpha=(p^2-1)(a_1+a_2p+\cdots)$.
	We begin with a technical numerical lemma.
	\begin{lem}\label{lemma: number}
	Let $\alpha \in \R$ and $\ell \in \mathbb{Z}$.  If $\ell-\bar{\alpha} \neq 0$, then for any $K>0$, there exists $k>K$ such that $\ell_k=sp^n$ where $p\nmid s$ and $n<k$.
	\end{lem}
	\begin{proof}
	We prove this by contradiction. Note that if an integer $a$ is not divisible by $p^k$, then $a=sp^n$ where $p\nmid s$ and $n<k$.  So assume the statement for $\ell_k$ is not true.  Then for all $k>K$, $\ell_k$ is divisible by $p^k$.  This implies that $\ell - \bar{\alpha}= 0$ and is a contradiction to the condition $\ell - \bar{\alpha} \neq 0$.
	\end{proof}
	The key observation is the following lemma.
	\begin{lem}{\label{stable}}
	The subgroup of $\alpha$-unstable elements in $\pi_* S^{2\alpha}/p$ is a finitely generated $\mathbb{F}_p[v_1]$-submodule.
		%Let $\alpha \in \R$.  Then except possibly finitely many elements $x_i$s, and elements of the form $v_1^m x_i$ for $m \in \mathbb{N}$, all elements are $\alpha$-stable.
	\end{lem}		
		
	\begin{proof}
	By definition \ref{defn:stable1}, we need to show there are only finitely many pair $(\ell,x) \in \mathbb{Z}\times B$ such that $(\ell, \alpha, x)$ is unstable.  Note that there are only 6 elements in $B$.  We discuss case by case on $x \in B$.  We prove that for a fixed $x \in B$, $\alpha \in \R$, there are only finitely many $\ell \in \mathbb{Z}$ such that $(\ell, \alpha, x)$ is unstable.  From now on, we can assume that $\ell + \alpha \neq 0$.  By Lemma \ref{lemma: number}, there exists some $k_0$ such that $p^{k_0}>|\ell|$, and $\ell_{k_0}=sp^n$ where $p\nmid s$, $n<k_0$.
 
	\begin{description}
	\item[Case 1] $x=1$\\
		Since $p^{k_0} \mid \ell_k-\ell_{k_0}$ for $k>k_0$, we have $\ell_k=s_kp^n$ where $p \nmid s_k$. When $k>k_0$, we have $k>n$ and $p^k>b_n$, so $max \{ 0, p^k - b_n \}=p^k-b_n$. Then $J_k(\ell, \alpha, 1)=\{p^k-b_n \leqslant j \leqslant p^k - 1\}$ and $\displaystyle\liminf_{k\rightarrow \infty}J_k(\ell, \alpha, 1)=\emptyset$.
	\item[Case 2] $x=h_0$\\
		\begin{enumerate}
		\item If $s\neq -1$ mod $p$, for $k>k_0$, we have $\ell_k=s_kp^n$ where $s_k \neq 0,\,-1$ mod $p$. In this case, we have $max \{ 0,\, p^k - A_n -2\}=p^k - A_n -2$. Then $J_k(\ell, \alpha, h_0)=\{p^k - A_n -2 \leqslant j \leqslant p^k - 1\}$ and $\displaystyle\liminf_{k\rightarrow \infty}J_k(\ell, \alpha, h_0)=\emptyset$.
		\item If $s=-1$ mod $p$, then we write $s=s'p-1$ and consider two sub cases.
			\begin{enumerate}
				\item If $p \nmid s' + a_{k_0+1}$, then for $k>k_0$, we have $\ell_k=s'_kp^{n+1}-p^n$ where $p \nmid s'_k$ and hence $J_k(\ell, \alpha, 1)=\{0 \leqslant j \leqslant \min\{b_{n+1}-1, p^k-1\}\}$.  Because $k>k_0>n$, we have $\min\{b_{n+1}-1, p^k-1\}=b_{n+1}-1$ which is independent of $k$.  Therefore, $(\ell, \alpha, h_0)$ is stable in this case.\\
				\item If $p \mid s' + a_{k_0+1}$, then for $k>k_0$, we have $\ell_k=s'_kp^{n+2}-p^n$ and $max \{ 0, p^k - (p^n-p^{n-2}+A_{n-2}+2)\}=p^k - (p^n-p^{n-2}+A_{n-2}+2)$. Then $J_k(\ell, \alpha, h_0)=\{p^k - (p^n-p^{n-2}+A_{n-2}+2) \leqslant j \leqslant p^k - 1\}$. In this case, we have $\displaystyle\liminf_{k\rightarrow \infty}J_k(\ell, \alpha, h_0)=\emptyset$.
			\end{enumerate}
		\end{enumerate}
	\item[Case 3] $x=h_1$\\
		\begin{enumerate}
		\item If $n>0$, for $k>k_0$, we have $max \{ 0, p^k-p+1\}=p^k-p+1$ and $J_k(\ell, \alpha, h_1)=\{p^k-p+1 \leqslant j \leqslant p^k - 1\}$.  We have $\displaystyle\liminf_{k\rightarrow \infty}J_k(\ell, \alpha, h_1)=\emptyset$.
		\item If $n=0$, for $k>k_0$, we have $J_k(\ell, \alpha, h_1)=\{0\}$ which is independent of $k$ and $(\ell, \alpha, h_1)$ is stable in this case.
		\end{enumerate}
	\item[Case 4] $x=g_0$\\
		\begin{enumerate}
		\item If $\ell_{k_0}=-1$ mod $p$, we have $\min\{p-2,\,p^k-1\}=p-2$, and then $J_k(\ell, \alpha, g_0)=\{0 \leqslant j \leqslant p-2\}$. Note that $J_k(\ell, \alpha, g_0)$ is independent of $k$ so $(\ell, \alpha, g_0)$ is stable in this case.
		\item If $\ell_{k_0}\neq-1$ mod $p$, then $J_k(\ell, \alpha, g_0)=\{p^k-1\}$ and $\displaystyle\liminf_{k\rightarrow \infty}J_k(\ell, \alpha, g_0)=\emptyset$.
		\end{enumerate}
	In the following Case 5 and Case 6, considering the condition that $\ell-\bar{\alpha} \neq \displaystyle\sum^\infty_{k>0}(p-2)p^k+(p-1)$, we can assume this condition because at most one $\ell \in \mathbb{Z}$ fails this condition.  With this condition, there exists a $k_0$ such that
		$$\ell-\bar{\alpha}=sp^{k_0}-p^{k_0-1}-p^{k_0-2}-\cdots-1~mod~p^{k_0+1}$$
		and $s\neq -1$ mod $p$.
	\item[Case 5] $x=g_1$\\  
		\begin{enumerate}
		\item If $s\neq0$ mod $p$, then for $k>k_0+1$, we have $\ell_k=s_kp^{k_0}-p^{k_0-1}-p^{k_0-2}-\cdots-1$ with $s_k \neq 0,\,-1$ mod $p$, and $\min\{A_{k_0}+1,p^k-1\}=A_{k_0}+1$.  Then $J_k(\ell, \alpha, g_1)=\{0\leqslant j \leqslant A_{k_0}+1 \}$, which is independent of $k$.  Therefore, $(\ell, \alpha, g_1)$ is stable in this case.
		\item If $s=0$ mod $p$ and $s= -p$ mod $p^2$, then for $k>k_0+1$, we have $\ell_k=s_kp^{k_0+2}-p^{k_0+1}-p^{k_0-1}-p^{k_0-2}-\cdots-1$, and $min\{p^{k_0+2}-p^{k_0}+A_{k_0}+1,p^k-1\}=p^{k_0+2}-p^{k_0}+A_{k_0}+1$. Then $J_k(\ell, \alpha, g_1)=\{0\leqslant j \leqslant p^{k_0+2}-p^{k_0}+A_{k_0}+1 \}$, which is independent of $k$.  Therefore, $(\ell, \alpha, g_1)$ is stable in this case.
		\item If $s=0$ mod $p$ and $s\neq -p$ mod $p^2$, then for $k>k_0+1$, we have $\ell_k=s_kp^{k_0+1}-p^{k_0-1}-p^{k_0-2}-\cdots-1$ with $s_k \neq -1$ mod $p$, and $J_k(\ell, \alpha, g_1)=\{\max\{0,\,p^k-b_{k_0+1}\} \leqslant j \leqslant p^k-1\}$. We have $\displaystyle\liminf_{k\rightarrow \infty}J_k(\ell, \alpha, g_1)=\emptyset$.
		\end{enumerate}
	\item[Case 6] $x=h_0g_1$\\
		 For $k>k_0+1$, we have $\ell_k=s_kp^{k_0}-p^{k_0-1}-p^{k_0-2}-\cdots-1$ with $s_k \neq -1$ mod $p$, and $\min\{b_{k_0}-1,p^k-1\}=b_{k_0}-1$. Then $J_k(\ell, \alpha, h_0g_1)=\{0\leqslant j \leqslant b_{k_0}-1 \}$, which is independent of $k$.  Therefore, $(\ell, \alpha, h_0g_1)$ is stable in this case.	
	\end{description}
	\end{proof}	
		
Lemma \ref{stable} reduces the question whether $S^{2\alpha}$ is of finite type to the question whether there are finitely many $\alpha$-stable $\mathbb{F}_p$-generators $v_1^jv_2^{\ell-2\alpha}x$ in a given bidegree $(s,t)$.

		We now divide elements of $H^*(C_0/v_1^{p^k},d)$ into subsets of $v_1$ towers.  First we divide the elements into subsets by rows, that is, we denote the subset of elements in a row by the name in the first column in Table \ref{element1}. For example, the first row in the table is $v_1^j$, $0 \leqslant j \leqslant p^k-1$, then we define a subset $(1,k)$ to be $\{ v_1^j \mid 0 \leqslant j \leqslant p^k-1\}$.  The name $1$ indicates that it consists of $v_1$ towers starting at $1$ and $k$ indicates that these elements are from $S^0/(p,v_1^{p^k})$. Similarly the name for the second row is $(h_0,k)$. For the third row, the subset $(v_2^{sp^n-p^{n-1}}h_0,k)$ can be divided into smaller subsets with respect to $n$. If $n_0$ is an integer, we denote 
		$$\{ x \in (v_2^{sp^n-p^{n-1}}h_0,k) \mid x \mbox{ is of the form } v_1^jv_2^{sp^{n_0}-p^{n_0-1}}h_0 \}$$
		 by $(v_2^{sp^{n_0}-p^{n_0-1}}h_0, k)$.  Then we have 
		 $$(v_2^{sp^n-p^{n-1}}h_0, k) = \overset{\infty}{\underset{m = 1}{\cup}} (v_2^{sp^m-p^{m-1}}h_0, k).$$
		
		%%\textcolor{red}{rewrite} We denote the following finite set by $N$
		%%$\{1, h_0, v_2^{sp^n-p^{n-1}}h_0, v_2^sh_1, v_2^{sp-1}g_0, v_2^{sp^n-\frac{p^{n-1}}{p-1}}g_1, \cdots\}$.
		%%\begin{defn}
		%%An element $z \in \pi_* S/ (p,v_1^{p^k})$ is of form $y$ if $x \in (y)_k$ where $y$ is $1$ or $h_0$ or $v_2^{sp^n-p^{n-1}}h_0$ or $\cdots$. Moreover, we say that $x$ is of form $y$ with $m$ if $x \in (y,m)_k$ where $m$ is an integer. 
		%%\end{defn}

			\begin{table}[H] 
			\centering
			\caption{Dividing the third row into subrows}
			\label{Dividing}
			\begin{tabular}{| p{3.6cm} | l |p{4cm} | p{5cm}|}
			\hline
 				Names & $Coker$ part		&	&	\\
			\hline
			\hline
				
				$(v_2^{sp^n-p^{n-1}}h_0,k)$ & $v_1^j v_2^{sp^n-p^{n-1}} h_0$	&	$ p \nmid s, ~ n \geqslant 1$	&	$ 0 \leqslant j \leqslant min \{ b_n - 1, p^k - 1 \}$ \\
				\hline
				\hline
			
			$(v_2^{sp^1-1}h_0,k)$ & $v_1^j v_2^{sp-1} h_0$	&	$ p \nmid s$	&	$ 0 \leqslant j \leqslant min \{ b_0 - 1, p^k - 1 \}$ \\
				\hline
						$(v_2^{sp^2-p}h_0,k)$ & $v_1^j v_2^{sp^2-p} h_0$	&	$ p \nmid s$	&	$ 0 \leqslant j \leqslant min \{ b_1 - 1, p^k - 1 \}$ \\
				\hline
				$\cdots$ \\
				\hline
			$(v_2^{sp^{n_0}-p^{n_0-1}}h_0,k)$ & $v_1^j v_2^{sp^{n_0}-p^{n_0-1}} h_0$	&	$ p \nmid s$	&	$ 0 \leqslant j \leqslant min \{ b_{n_0} - 1, p^k - 1 \}$ \\
				\hline
				$\cdots$ \\
				\hline
					
			\end{tabular}
		\end{table}

\begin{ex}		
		The element $v_1^3v_2^{2p-1}h_0$ is in $(v_2^{sp^1-1}h_0,k) \subset (v_2^{sp^n-p^{n-1}},k)$.  This is because $v_1^3v_2^{2p-1}h_0$ belongs to the third row in the $Coker$ part.  The power of $v_2$ is $2p-1=2p^1-p^{1-1}$ so it is in $(v_2^{sp^1-1}h_0,k)$.
\end{ex} 
%If the general format of the row does not have $n$ in its name, then we say any element of this form is with $0$. For example, $v_1^2$ is of form $1$ with $0$ because the general format is $v_1^j$ and the base is $1=v_2^0$ with no $n$ in its name.

We have divided elements in $H^*(C_0/v_1^{p^k},d)$ into subsets by rows, and each row may divide into smaller subsets, which we will call subrows.  The length of the $v_1$-tower is determined by which (sub)rows the element lie in.  Now we will group elements in $\pi_* S^{2\alpha}/p$ into (sub)rows in a similar way, which gives an alternative description of $\alpha$-stable.
			
	\begin{defn}\label{defn:stable}	
		Given an $\alpha \in \R$, an element $v_1^m v_2^{\ell-\alpha} x \in \pi_*S^{2\alpha}/p$ is \emph{$\alpha$-stable} if there exists $k_0 \in \Z^+$ and for all $k>k_0$, we have the elements $v_1^m v_2^{\ell-\bar{\alpha}_k}x \in \pi_*S^{2\alpha_k}/p,v_1^{p^k}$ are in the subset $(y,k)$ where $y$ is fixed. Otherwise, we will call the element \emph{$\alpha$-unstable}.
	\end{defn}
	
	It is straightforward to check that Definition \ref{defn:stable} is equivalent to Definition \ref{defn:stable1}.

	\begin{ex}
	Let $\alpha$ be an integer $n \in \R$.  Then all elements $v_1^m v_2^{\ell-2n} x$ are $n$-stable by definition.
	\end{ex}
	
	We are trying to decide if $S^{2\alpha}/p$ is of finite type.  By Lemma \ref{stable}, we need only to focus on stable elements.
	
	\begin{defn}\label{defn:form}
	We define subset $(y)$ of the $\alpha$-stable elements by
	$$(y)=\{x=\lim_{k\rightarrow \infty} x_k \in \pi_*S^{2\alpha}/p \mid x \text{ is $\alpha$-stable, } x_k \in (y,k) ~ when ~k ~is~large\}. $$
	 %$v_1^mv_2^{\ell - \alpha} x$ into subsets $(y)$ where 
	%of a certain form with some number $n_0$ if there exists $k_0$, such that for all $k>k_0$, we have $v_1^mv_2^{\ell - \bar{\alpha}_k}x$ is of that form with the same $n_0$.
	 \end{defn}
	 By the definition of $\alpha$-stable, all stable elements lie in some subsets defined above.
	 \begin{ex}
	 In $\pi_ {-1}S^{2\lambda}$, the element $h_0v_2^{-\lambda}$ is in $(h_0v_2^{sp^1-1})$. The reason is as follows. When $k>1$, we have $h_0v_2^{-\lambda_k}=h_0v_2^{-\frac{p^{k-1}-1}{p-1}p-1}$ in $(h_0v_2^{sp^1-1},k)$.  By Definition \ref{defn:stable}, the element $h_0v_2^{-\lambda}$ is $\lambda$-stable and by Definition \ref{defn:form}, it is in $(h_0v_2^{sp^1-1})$.
	 \end{ex}
	 
	Because the following lemma, we will focus on the $Coker$ part. 	
	\begin{lem}\label{lem:kernel}
	No stable element in the $Ker$ part will survive in the limit.
		%For any nontrivial element $y=\lim_k y_k \in \pi_*S^\alpha/p$, $y_k \in H^*(C_0/v_1^{p^k},d)$, there are only finitely many $k$ such that $y_k$ is in the $Ker$ part of $H^*(C_0/v_1^{p^k},d)$.
	\end{lem}
		
	\begin{proof}
		We can check this row by row. The results in the table \ref{element1} shows that for any element in the $Ker$ part, its $v_1$ tower's range does not overlap when $k$ goes to infinity.  For example, non trivial elements in the second row of the $Ker$ part, of the form $v_1^jv_2^{sp^n}$, has $j$ in the range $max \{ 0, p^k - b_n \} \leqslant j \leqslant p^k - 1$.  When $k>>n$, the range is $A(k) \coloneqq  \{ j \in \Z \mid 0, p^k - b_n \leqslant j \leqslant p^k - 1 \}$, and $A(k) \cap A(k+1) = 0$, so no element survives.	
	\end{proof}

	If $\pi_* S^{2\alpha}$ is not finite at some stem, then there exists a bidegree $(s,t)$ such that $H^{s,t}(S^{2\alpha})$ is not finite. If this happens, because there are only finitely many subsets, there must be a subset so that infinitely many elements of this subset survive in the limit at this bidegree.

%We can focus on the $H^s(S^\alpha)/\zeta$ where $\zeta \in H^{1,0}(S^0)$ is defined by the cross homomorphism 
%$$\mathbb{G}_2 \rightarrow \mathbb{Z}_p,$$
%which sends $g\in \mathbb{G}_2$ to $\frac{1}{p}log(1+p(g'(0)^{-(p+1)})det(g)$.
We check each row for the possibility to have infinitely many elements survive in the limit at some bidegree.  Some forms can be easily excluded.
	
	\begin{lem}\label{lem:form}
		There are only finitely many elements in the subsets $(1)$, $(h_0)$, $(v_2^sh_1)$, $(v_2^{sp-1}g_0)$ that survive at a fixed degree $(s_0,t_0)$.
	\end{lem}
	
	\begin{proof}
	From the table \ref{element1}, if the length of the $v_1$ tower in this form is bounded by a finite number, then there are only finite elements of a fixed bidegree. For example, for elements in $(v_2^sh_1)$, the length of the $v_1$ tower is bounded by $1$.
	\end{proof}

	Now we will focus on the rest rows that possibly have infinite length $v_1$ towers. By the Lemma \ref{lem:form} and Lemma \ref{lem:kernel}, there are only four possible rows: $(v_2^{sp^n-p^{n-1}}h_0)$, $(v_2^{sp^n-\frac{p^n-1}{p-1}}g_1)$, $(v_2^{sp^n-p^{n-1}-\frac{p^n-2}{p-1}}g_1)$ and $(v_2^{sp^n-\frac{p^n-1}{p-1}}h_0g_1)$. We shall analyse, row by row, the conditions on $\alpha$ such that at some $(s,t)$, there are infinitely many elements of one of the four rows that survive in the limit. Since shifting by an integer degree will not change the property of finiteness, we assume $a_0=0$ from now on.
	
	%For each $n$,  there are only finitely many possibilities of a given form with $n$ at a fixed bidegree, so if $H^{(s,t)}(S^\alpha)$ is not finite, we will have that there is at least one form, such that there are infinitely many elements of this form that survive; in particular, they must be of this form with $n$ and $n$ must go to infinity.
	
	In each case, the question of whether $\pi_k S^{2\alpha}/p$ has infinitely many elements in a certain subset at some stem $k$ reduces to an elementary numerical question.  %Note that $X$ is of finite type if and only if $\sum^n X$ is of finite type for some integer $n$.%  Without loss of generality, we can pick the inner degree $t$ with infinitely many elements.  In the following lemmas, we will pick $t=2(p-1)$.
	
%%	For example, the first form $v_2^{sp^n-p^{n-1}}h_0$ case, we end up with the following question. 
	
%%	\begin{que}
%%		Given an $\alpha \in \lim \Z / 2p^k(p^2-1)$, is there an integer $m$, and infinite integers $l_j>0$, $n_j$ such that 
%%		$$m-\alpha_k = s_kp^n-p^{n-1}$$ with $p \nmid s_k$, for $k$ large
%%		$$m-l_j-\alpha_k = s_{k,j}p^{n_j}-p^{n_j-1}$$ with $p \nmid s_{k,j}$, for $k$ large
%%		and
%%		$$(p+1)l_j \leqslant p^{n_j-1}(p+1)-2$$
%%	\end{que}
	We shall start with the row $(v_2^{sp^n-p^{n-1}}h_0)$.

\begin{thm}{\label{case1}}
		If there are infinitely many elements in $H^{1,2t(p-1)}(E_2 S^{\alpha}/p)$ in $(v_2^{sp^n-p^{n-1}}h_0)$ then in the expansion of $\alpha$, infinitely many $a_i$s are nonzero and infinitely many $a_i$s are zero.  Moreover, the converse statement is also true.
	\end{thm}
	
	\begin{proof}
		Let $\ell_0$ be the unique integer in $(\frac{t-1}{p+1}-1,\frac{t-1}{p+1}]$.  Then $|v_2^{\ell_0}h_0| \leqslant 2t(p-1) < |v_2^{\ell_0+1}h_0|$.  At the bidegree $(1,2t(p-1))$, elements with the base $h_0$ are $v_1^{(p+1)\ell+j_0} v_2^{\ell_0-\ell - \alpha} h_0$ where $\ell \in \mathbb{N}$ and $j_0=t-1-(p+1)\ell_0$.  By the assumption, we can assign infinitely many $m \in \Z^+$ a distinct nonzero element $v_1^{(p+1)\ell_m+j_0} v_2^{\ell_0-\ell_m - \alpha} h_0$ in $(v_2^{sp^n_m-p^{n_m-1}}h_0) \subset (v_2^{sp^n-p^{n-1}}h_0)$ where $n_m \in \Z$.  We can assume that if $m>m'$, then $\ell_m>\ell_{m'}$.  This gives following two restrictions on $\ell_m$ and $\alpha$.
		 \begin{enumerate}
		 	\item $\ell_0-\ell_m - \alpha \equiv s_m p^{n_m}-p^{n_m-1}$ mod $p^{n_m+1}$ with $p \nmid s_m$
		 	\item $(p+1)\ell_m+j_0 \leqslant p^{n_m-1}(p+1)-2$
		  \end{enumerate} 
		
		  The the first restriction comes from the assumption that those elements are in $(v_2^{sp^n_m-p^{n_m-1}}h_0)$, that is, the power of $v_2$ is of the form $sp^n_m-p^{n_m-1}$; the second restriction comes from the bound on the length of $v_1$ towers: since nontrivial elements in  $(v_2^{sp^n_m-p^{n_m-1}}h_0)$ are $\{v_1^jv_2^{sp^n_m-p^{n_m-1}}h_0\}$ with $0 \leqslant j \leqslant p^{n_m-1}(p+1)-2$, if $v_1^{(p+1)\ell_m+j_0} v_2^{\ell_0-\ell_m - \alpha} h_0 \neq 0$, then we have the second restriction.  From the assumption, we have $\ell_m>0$; from the second restriction, we have $l_m \leqslant p^{n_m-1}$.  The first restriction tells us
		  
		  $$\ell_m  + \alpha_{n_m} = p^{n_m-1}+\ell_0$$
		  
		  Plugging $\ell_m = p^{n_m-1} + \ell_0 - \alpha_{n_m}$ into the restriction $0< \ell_m < p^{n_m-1}$, we have $\alpha_{n_m}<p^{n_m-1}+\ell_0$.  There exists $n_M>log_p\ell_0+1$.  The condition $p \nmid s_m$ shows that $a_{n_M+1} \neq 0$.  Recall that $\alpha_{n_m} = \displaystyle\sum^{n_m}_{i=1} a_i p^{i-1}$ and $0 \leqslant a_i < p$.  When $m>M$,  $p^{n_M}+a_{n_m}p^{n_m-1} \leqslant \alpha_{n_m}<p^{n_m-1}+\ell_0$ is equivalent to $a_{n_m} = 0$.  The condition $p \nmid s_m$ shows that $a_{n_m+1} \neq 0$. So the two restrictions are equivalent to $a_{n_m} = 0$ and $a_{n_m+1} \neq 0$.  Hence, for each $m>M \in \Z^+$, we have $a_{n_m} = 0$ and $a_{n_m+1} \neq 0$, and there are infinitely many are nonzero coefficients and there are infinitely many zero coefficients in the $\R$-expansion of $\alpha$.
		  
The above proof shows that the converse statement is true.  In fact, if there are infinitely many $a_i$s are nonzero and there are infinite $a_i$s are zero, then we can assign each $m \in \Z^+$ a different number $n_m \in \Z^+$ such that $a_{n_m}=0$ and $a_{n_m+1} \neq 0$. Let $l_m = p^{n_m-1} - \alpha_{n_m}+\ell_0$, then check the table, we will have $v_1^{(p+1)\ell_m+j_0} v_2^{\ell_0-\ell_m - \alpha} h_0$ are survived elements.  Therefore, we have infinitely many elements in $(v_2^{sp^n-p^{n-1}}h_0) \subset H^{1,2t(p-1)}(E_2 S^{2\alpha}/p)$.
	\end{proof}
%%%%%%	

	%\begin{thm}
	%	If there are infinitely many elements in $H^{1,t_0}(E_2 S^{\alpha}/p)$ in $(v_2^{sp^n-p^{n-1}}h_0)$ at some bidegree $(1, t_0)$, then in the $\R$-expansion coefficients of $\alpha$, infinitely many $a_i$s are nonzero and infinitely many $a_i$s are zero.  Moreover, the converse statement is also true.
	%\end{thm}
	
	%\begin{proof}
	%	This follows from Lemma \ref{case1}.  If $\alpha$ satisfies the above condition, then $\alpha - t_0$ satisfies the condition in Lemma \ref{case1}, so by Lemma \ref{case1}, in the $\R$-expansion coefficients of $\alpha - t_0$, there are infinitely many $a_i$s are nonzero and there are infinitely many $a_i$s are zero. Notice that shifting by an integer will not change the property of whether there are infinitely many $a_i$s are nonzero and there are infinitely many $a_i$s are zero in the expansion. Thus, we know $\alpha$ also have the property. The converse part is similar.
	%\end{proof}
	
	%\begin{cor}
	%	If there are infinitely many elements in $H^{1,t_0}(E_2 S^{\alpha}/p)$ in $(v_2^{sp^n-p^{n-1}}h_0)$ at some bidegree $(1, t_0)$, then at all bidegree $(1, t_0+2k(p^2-1))$ $(k \in \Z)$ there are infinitely many elements in $H^{1,t_0+2k(p^2-1)}(E_2 S^{\alpha}/p)$ of the form $v_2^{sp^n-p^{n-1}}h_0$.
	%\end{cor}
	Theorem \ref{case1} shows that if $S^{2\alpha}$ is not of finite type, then the homotopy groups are not finitely generated in all possible nontrivial degrees.  We state it as Corollary \ref{cor:infinite}
	 \begin{cor}\label{cor:infinite}
	If there are infinitely many elements in $H^{1,t_0}(E_2 S^{2\alpha}/p)$ in $(v_2^{sp^n-p^{n-1}}h_0)$ at some bidegree $(1, t_0)$, then at all bidegree $(1, 2t(p-1))$ $(t \in \Z)$ there are infinitely many elements in $H^{1,t_0+2k(p^2-1)}(E_2 S^{2\alpha}/p)$ of the form $v_2^{sp^n-p^{n-1}}h_0$.
	\end{cor}

	 With the same approach, one could do with the elements of other forms. We state the results as follows.

	 \begin{thm}
		If there are infinitely many elements in 
		$$(v_2^{sp^n-\frac{p^{n-1}-1}{p-1}}g_1), ~or$$
		$$(v_2^{sp^n- p^{n-1} - \frac{p^{n-2}-1}{p-1}}g_1), ~or$$
		$$(v_2^{sp^n- \frac{p^n-1}{p-1}}h_0g_1)$$
		at some bidegree $(s,t)$, then in the $\R$-expansion coefficients of $\alpha$, infinitely many $a_i$s are nonzero and infinitely many $a_i$s are zero.  Moreover, the converse statement is also true.
	\end{thm}
	
	Summing up all the cases together, we have the main theorem as follows.
		
	\begin{thm}\label{condition}
		If there are infinitely many elements in $H^{s,t_0}(E_2 S^{2\alpha}/p)$ at some bidegree $(s, t_0)$, then in the expansion of $\alpha$, infinitely many $a_i$s are nonzero and infinitely many $a_i$s are zero.  Moreover, the converse statement is also true.
	\end{thm}
	 
 	Theorem \ref{reduce} and Theorem \ref{condition} answer the Question \ref{question}. %To state the answer, we introduce the following notations. 
%	\begin{defn}
%		Let $f$ be a $\Zp$-module map $f$ from $Pic_{K(2)} \cong \Zp \oplus \Zp \oplus \Z/2(p^2-1)$ to $\lim \Z/2(p^2-1)p^k$ that sends $(1,0,0)$ to $(p^2-1)$,	$(0,1,0)$ to $p+1+ \displaystyle\sum^\infty_{k=0}2(p^2-1)p^k$, and $(0,0,1)$ to $p+1+ \displaystyle\sum^\infty_{k=0}2(p^2-1)p^{2k}$.	
%	\end{defn}
%	Recall that for any element $\alpha \in \lim \Z/2(p^2-1)p^k$, there is a unique expansion $\alpha=a_0+2(p^2-1)\displaystyle\sum^\infty_{i=1}p^{i-1}a_i$. Therefore, one can define $a_i \colon Pic_{K(2)} \rightarrow \Z/p$ for $i \geqslant 1$ by $a_i(X)=a_i$ where $a_i$ is the $(i+1)$\textsuperscript{th} coefficient in the expansion of $f(X)$.
	\begin{thm}\label{result}
	For any $X \in Pic^0_{K(2)}$, let $e_i(X)$ be the $i$\textsuperscript{th} coefficient in the $\R$-expansion of $e(X)$.  Then $\pi_kX$ is finitely generated for all degrees $k \in \Z$ if and only if either only finitely many $e_i(X)$'s are zeros, or only finitely many $e_i(X)$'s are nonzeros.
	\end{thm}
	\begin{proof}
	From Theorem \ref{reduce}, $\pi_kX$ is finitely generated for all degrees $k \in \Z$ if and only if $\pi_k S^{e (X)}/p$ is.  Then the result follows from the contrapositive of Theorem \ref{condition}.
	\end{proof}
	
At the end of this section, we state a Corollary of Lemma \ref{stable}.

\begin{cor}
Let $X$ be an element in $Pic_{K(2)}$.  Then the set of $v_1$-free elements in $\pi_*X$ is finitely generated as a $\mathbb{Z}_p[v_1]$-module.
\end{cor}
\begin{proof}
We first show the statement for $X/p \simeq S^{2e(X)}/p$.  All $v_1$-free elements have infinite $v_1$-towers.  The length of $v_1$-tower on an element $x$ in $\pi_*S^{2e(X)}/p$ is determined by the form and level of $x$ if $x$ is $e(X)$-stable.  By Lemma \ref{stable}, the set of unstable elements in $\pi_*S^{2e(X)}/p$ is finitely generated as an $\mathbb{F}_p[v_1]$-module.  Except for the form $1$ and $h_0$ ($\zeta$ and $\zeta h_0$), all other forms have finite length $v_1$-towers and are $v_1$-torsion.  The length of $v_1$-tower on a stable element is the same as the length of $v_1$-tower of the form that this element stables to.  The $v_1$-towers are finite on all stable elements but at most finitely many exceptions ($1$, $h_0$, $\zeta$, and $\zeta h_0$).  Therefore, the $v_1$-free elements must belong to the finitely many $v_1$-towers in those exceptional cases or be unstable elements.  So the set of $v_1$-free elements in $\pi_*X/p$ is finitely generated as an $\mathbb{F}_p[v_1]$-module.  In particular, this is also true for $\pi_*X/(p\pi_*X)$.  Since $\pi_* X$ is $p$-complete, this implies that the set of $v_1$-free elements in $\pi_*X$ is finitely generated as a $\mathbb{Z}_p[v_1]$-module.

\end{proof}

%%%%%%
%Edited after this line

%%%%%%
\section{examples}\label{examplesection}
	In this section, we examine Theorem \ref{result} with three examples: $L_{K(2)}S^0$, $I_2S^0$ and $L_{K(2)}S^{2\gamma}$ where $I_2S^0$ is the Gross-Hopkins dual of the $K(2)$-local sphere (see \cite{MR3436395} for the definition of $I_2$) and $\gamma = \displaystyle\lim_k p^{2k} \in \R$ as before is a generator of the torsion part in $Pic_{K(2)}$. 
	\begin{enumerate}
	\item For $X=L_{K(2)}S^0$, $e(X)=0$ and $e_i(X)=0$ for all $i$ and there are only finitely many nonzero $e_i(X)$s.  Theorem \ref{result} implies that $L_{K(2)}S^0$ satisfies the finitely generated property, which agrees with the known computation.
	\item For $X=I_2S^0$, in large prime case, $I_2S^0=S^{n^2-n}\wedge S[det]$.  Since integer shifts does not change the finitely generated property, it is enough to consider $S[det]$.  From Theorem \ref{thm:relation}, we have $e(S[det])=\lambda=(p+1)+ \displaystyle\sum^\infty_{k=0}(p^2-1)p^k$ and $e_i(X)=1$ for all $i\geqslant 1$.  Therefore, there are only finitely many zero $e_i(X)$s. Theorem \ref{result} implies that $L_{K(2)}S^0$ satisfies the finitely generated property.
	\item For $X=L_{K(2)}S^{2\gamma}$, $e(X)=1+ \displaystyle\sum^\infty_{k=0}(p^2-1)p^{2k}$.  There are infinitely many zero $e_i(X)$s (when $i\geqslant 1$ is odd) and infinitely many nonzero $e_i(X)$s (when $i\geqslant 1$ is even).  Theorem \ref{result} implies that $\pi_k L_{K(2)}S^{2\gamma}$ is not finitely generated for some stem $k$.  In fact, in $\pi_{-2p^3+2p^2+4p-7}S^{2\gamma}$, we have linearly independent elements $v_1^{j_k}v_2^{m_k-\gamma}h_0$ for all $k \geqslant 0$ where $m_k = -p^{2k+1}+\frac{p^{2k+2}-1}{p^2-1}$, $j_k=(p+1)(m_k-m_0)$.
	\end{enumerate}
	
As an application, we have the following theorem about Gross--Hopkins duality at prime $p\geqslant 5$, height $2$.  The (non-local) Brown--Comenetz dual of the sphere $I_{\mathbb{Q}/\mathbb{Z}}$ is the spectrum representing the generalized cohomology theory
$$X \rightarrow Hom(\pi_{-*}X, \mathbb{Q}/\mathbb{Z}).$$
The (non-local) Brown--Comenetz dual of a spectrum $X$ is defined to be $I_{\mathbb{Q}/\mathbb{Z}}(X) \coloneqq F(X,I_{\mathbb{Q}/\mathbb{Z}})$.  However, if we start with a $K(n)$-local spectrum $X$, the Brown--Comenetz dual $I_{\mathbb{Q}/\mathbb{Z}}(X)$ may not be $K(n)$-local any more.  The Gross--Hopkins dual is a $K(n)$-version Brown--Comenetz dual (see \cite{MR2946825} and \cite{bbs_gross}).

\begin{defn}\label{defn:GrossHopkins}
Let $L_nX$ be the localization of $X$ with respect to the $n$\textsuperscript{th} Morava $E$-theory. Let $M_nX$ be the $n$\textsuperscript{th} monochromatic layer of $X$; that is, the fiber of $L_n X \rightarrow L_{n-1}X$.  The height $n$ Gross--Hopkins dual of $X$ is defined to be
$$I_nX \coloneqq F(M_nX, I_{\mathbb{Q}/\mathbb{Z}}).$$
Denote $I_n S^0$ by $I_n$.
\end{defn}  

While $I_nX$ is automatically $K(n)$-local (\cite[Proposition~2.2]{MR2946825}), the trade off is that it is usually very hard to compute $\pi_*I_nX$ from $\pi_*X$.  If $X \in Pic_{K(n)}$, then $I_nX=X^{-1}\smash I_n \in Pic_{K(n)}$.  Because $I_n$ is dualizable in $K(n)$-local category and we have $I_n X= D_nX \smash I_n$ where $D_nX$ is $F(X,L_{K(n)}S^0)$.   As an application of Theorem \ref{thm:mainresult}, we show the following theorem.

\begin{thm}\label{thm:ghdual}
Let $I_2$ be the Gross--Hopkins dual at prime $p\geqslant 5$, height $2$, $X \in Pic_{K(2)}$, and $\lambda=\displaystyle\lim_k p^{2k}(p^2-1)\in \R$. Then $e(I_2X)=2+\lambda-e(X)$.  In particular, $X$ is of finite type if and only if $I_2X$ is of finite type.
\end{thm}

\begin{proof}%[Proof of Theorem \ref{thm:ghdual}]
The statement $e(I_2X)=2+\lambda-e(X)$ follows from the facts:
\begin{enumerate}
\item $e$ is a group homomorphism;
\item there is an equivalence
$$I_2 X \simeq D_2X \wedge I_2$$
where $D_2 X$ is the $K(2)$-local Spanier--Whitehead dual of $X$ and $I_2$ is the Gross--Hopkins dual of the $K(2)$-local sphere;
\item $e(D_2X) = -e(X)$;
\item $I_2 \simeq \Sigma^{2^2-2}S[det]$.
\end{enumerate}
We have 
\begin{align*}
e(I_2X) &=e(D_2X \wedge \Sigma^{n^2-n}S[det])=e(D_2X)+2^2-2+e(S[det]) \\
& = -e(X) + 2+ \lambda.
\end{align*}
Next we will show that $X$ is of finite type if and only if $I_2X$ is of finite type.  By Theorem \ref{thm:mainresult}, this is equivalent to the statement that $e(X)$ has the finiteness property if and only if $e(I_2X)=2+\lambda-e(X)$ has the finiteness property.  We can ignore the integer shift $2$ when considering finiteness property.  Because $I_2(I_2 X)=X$, we only need to show one direction.  Assume that $X$ is of finite type, we will prove $I_2(X)$ is of finite type.  The rest is an elementary numerical analysis.
\begin{case}
$e(X)$ has finitely many nonzero entries in its $\R$-expansion coefficients. If $e(X)=0$, then $e(I_2X)=\lambda$.  This has finitely many zero entries in its coefficients.  If $e(X)\neq 0$, then the coefficients of $-e(X)$'s $\R$-expansion $e_k(D_2X)$ will always be $p-1$ when $k>K_0$ for some $K_0 \in \mathbb{Z}$.  The coefficients of $\lambda$ are always $1$.  So the $i$\textsuperscript{th} coefficients of $\lambda-e(X)$'s $\R$-expansion will be always be $1$ when $k>K_0+1$.  Then $e(I_2X)$ has finitely many zero entries.
\end{case}
\begin{case}
$e(X)$ has finitely many zero entries in its $\R$-expansion coefficients.  Note that $e(X)+e(I_2X)=\lambda$ and the coefficients of the $\R$-expansion of $\lambda$ are $1$.  We argue by contradiction.  If there are infinitely many zero entries and infinitely many non zero entries in the coefficients of the $\R$-expansion of $e(I_2X)$, then there are infinitely many places where the nonzero coefficient followed by a zero one.  Let $m$\textsuperscript{th} coefficient of $e(I_2X)$ be one of such places; i.e., $e_m(I_2X)\neq 0$ and $e_{m+1}(I_2X)=0$.  Now considering the sum $e(X)_m+e(I_2X)_m$, we have two cases.  If in the sum $e(X)_m+e(I_2X)_m \geqslant 2p^m(p^2-1)$, then at the $m+1$\textsuperscript{th} coefficients of the equation 
$$e(X)+e(I_2X)=\lambda,$$
we have
$$e_{m+1}(X)+e_{m+1}(I_2X)+ 1=1 ~or~ p+1.$$
In this case, $e_{m+1}(I_2X)=0$ implies that $e_{m+1}(X)=0$.

If in the sum $e(X)+e(I_2X)<2p^m(p^2-1)$, then we have
$$e_m(X)+e_m(I_2X)<p.$$
From the equatoin
$$e(X)+e(I_2X)=\lambda,$$
we have
$$e_m(X)+e_m(I_2X)+\epsilon=1 ~or~ p+1$$
where $\epsilon=0$ or $1$.  So we have
$$e_m(X)+e_m(I_2X) \leqslant 1.$$
The condition $e_m(X) \neq 0$ implies that $e_m(X) = 1$ and $e_m(X) = 0$.
\end{case}
In both cases, one such places would force a zero in $e(X)$.  This would imply that there infinitely many zero entries in $e(X)$, contradicted to the assumption.
\end{proof}
	
	\begin{rem}
	We know $L_{K(2)}S^0$ is of finite type from the computation. We would like to have some features that may be generalized to higher height cases. During the computation, one observation is that in the computation of $\pi_*L_{K(2)}S^0$, for all elements $x$ with $|x|=(t,s)$, $t-s<-1$ in the $E_2$ page of the ANSS, the $v_1$ tower on $x$ can not pass the line $t-s=1$, i.e., $v_1^kx=0$ for $k|v_1|+t-s \geqslant 0$. This phenomenon together with the symmetric of the $E_2$ page from Gross-Hopkins duality implies the finitely generated property of $\pi_*L_{K(2)}S^0$. There might be a conceptual algebraic argument of this phenomenon that works for higher heights.
	\end{rem}

	%%$\alpha_0 \equiv 2+2(p^2-1)+2p^2(p^2-1) + \cdots + 2p^{2k}(p^2-1)$ mod $p^{2k+1}(p^2-1)$ for all $k \geqslant 0$. We choose $\alpha_0$ because $L_{K(2)}S^{\alpha_0}$ is a torsion element of order $p^2-1$ in the Picard group.

	%%From the , we find that $\pi_*L_{K(2)}S^0$ and $\pi_*I_2S^0$ are finitely generated, $\pi_*L_{K(2)}S^{\alpha_0}$ is not finitely generated, in particular, we have $\pi_{2(p-1)(1-p-p^2)-3}S^{\alpha_0}$ have linearly independent elements $v_1^{j_k}v_2^{m_k-\alpha_0}h_0$ for all $k \geqslant 0$ where $m_k = -p^{2k+1}-\frac{p^{2k}-1}{p^2-1}$, $j_k=(p+1)(m_k-m_0)$.
\bibliographystyle{plain}
\bibliography{dual}

\end{document}